\documentclass[12pt]{article}
\usepackage{graphicx,color,amsmath,amssymb,amsthm,rotating,stmaryrd,datetime,pifont,booktabs,enumitem}
\usepackage[mathscr]{eucal}

\setlist{nosep}

\newtheorem{theorem}{Theorem}

\newtheorem{result}[theorem]{Result}
\newtheorem{lemma}[theorem]{Lemma}
\newtheorem{corollary}[theorem]{Corollary}

\newtheorem{defn}[theorem]{Definition}

\newcommand{\Label}{\label}

\newcommand{\Labele}{\label}

\newcommand\red[1]{{\color{red} #1}}

\definecolor{eggplant}{rgb}{0.68, 0.65, 0.32}

\makeatletter
\@ifpackageloaded{stix}{%
}{%
  \DeclareFontEncoding{LS2}{}{\noaccents@}
  \DeclareFontSubstitution{LS2}{stix}{m}{n}
  \DeclareSymbolFont{stix@largesymbols}{LS2}{stixex}{m}{n}
  \SetSymbolFont{stix@largesymbols}{bold}{LS2}{stixex}{b}{n}
  \DeclareMathDelimiter{\lDoubleBrace}{\mathopen} {stix@largesymbols}{"E8}%
                                            {stix@largesymbols}{"0E}
  \DeclareMathDelimiter{\rDoubleBrace}{\mathclose}{stix@largesymbols}{"E9}%
                                            {stix@largesymbols}{"0F}
}
\makeatother
\newcommand\figg[1]{\lDoubleBrace\hspace*{-0.12em} #1 \hspace*{-0.12em}\rDoubleBrace}

\newcommand{\Xone}{{$\underbrace{\phantom{XXXXXXlXXX}}_{\substack{\S_\theta,\ \theta\,\in\, \mathbb F_{\hspace*{-.5mm}{q^{3}}}^{\mystrut^*}}}$}}
\newcommand{\Xtwo}{{$\underbrace{\phantom{XXXXXXXlXXX}}_{\substack{\S_\theta, \ {{\textup{\textsf{N}}}}(\theta) \textup{ a nonzero}\\\textup{ square in } \Fq}}$}}
\newcommand{\Xthree}{{$\underbrace{\phantom{XXXXlXXX}}_{\substack{\S_\theta, \ {{\textup{\textsf{N}}}}(\theta) \textup{ a nonzero}\\\textup{ square in } \Fq}}$}}
\newcommand{\Sminusone}{\S_{\mbox{\tiny{-$1$}}}}
\newcommand{\Sone}{\S_{\mbox{\tiny{$1$}}}}
\newcommand{\figpiq}{$\mathcal P_{\!\!\!\phantom{.}_{\!{2\!,\!q}}}$}

\newcommand{\X}{\mathcal X}

\renewcommand{\S}{\mathbb{S}}

\newcommand{\K}{\mathcal{K}}

\newcommand{\E}{\mathcal{E}}

\newcommand{\D}{\mathcal{D}}

\newcommand{\F}{\mathcal{F}}
\newcommand{\PG}{{\textup{PG}}}
\newcommand{\PGL}{{\textup{PGL}}}
\newcommand{\PGammaL}{\textup{P}\Gamma{\textup {L}}}

\newcommand{\piq}{\mathcal P_{2,q}}

\newcommand{\III}{\textup{III}}
\newcommand{\II}{\textup{II}}
\newcommand{\I}{\textup{I}}

\newcommand\Norm{{\mbox{\textup{\textsf{N}}}}}
\newcommand\orb{{\textsl{orbits}}}

\renewcommand\Pr{\mbox{\footnotesize\texttt{\textup{Proj}}}}
\newcommand\Sp{\mbox{\footnotesize\texttt{\textup{Splash}}}}

\newcommand{\mupt}{\mu_{{\textup{\tiny{pt}}}}}
\newcommand{\muline}{\mu_{{\textup{\tiny{line}}}}}

\newcommand{\SG}{{\mathsf S_{\mbox{\tiny T}}}}

\newcommand{\Fqqq}{\mathbb{F}_{\hspace*{-1mm}{q^{3}}}}
\newcommand{\Fq}{\mathbb{F}_{\hspace*{-.5mm}q}}
\newcommand\mystrut{\rule{0pt}{1pt}}
\newcommand\Fqstar{\mathbb F_{\hspace*{-.5mm}q}^{\mystrut^*}}
\newcommand\Fqqqstar{\mathbb F_{\hspace*{-1mm}{q^{3}}}^{\mystrut^*}}

\newcommand{\ttl}{t}

\begin{document}

\title{The partition of $\PG(2,q^3)$ arising from an order 3 planar collineation}

\author{S.G. Barwick, Alice M.W. Hui and Wen-Ai Jackson}
\date{}
\maketitle



\begin{abstract}
Let $\phi$ be a collineation of order 3 acting on $\PG(2,q^3)$
whose fixed points are exactly an $\Fq$-plane $\piq$. Let $T$ be a  point whose  orbit  under $\phi$ is a triangle and let $\SG$ be the subgroup of $\PGL(3,q^3)$ that fixes setwise the $\Fq$-plane $\piq$ and fixes setwise  the line $T^\phi T^{\phi^2}$.
The point orbits of $\SG$ form a partition of the points of $\PG(2,q^3)$ and consist of:  the singletons $T,T^\phi, T^{\phi^2}$; scattered linear sets on the sides of the triangle $T T^\phi T^{\phi^2}$; and $\Fq$-planes.
 This article studies the structure of this partition, looking at maps that permute elements of the partition.
  The motivation in studying  this partition lies in its application to the construction of the Figueroa projective plane, and the article concludes with a characterisation in this setting.
\end{abstract}

%

\section{Introduction}\Label{388}

Denote the unique finite field of prime power order $q$ by $\Fq$, and let $\Fqstar=\Fq\setminus\{0\}$. An $\Fq$-\emph{plane}
of $\PG(2,q^3)$ is a subplane of order $q$, and is isomorphic to $\PG(2,q)$.
Let
$\phi\in\PGammaL(3,q^3)$ be a collineation of order 3 acting on the points of $\PG(2,q^3)$
whose fixed points are exactly an $\Fq$-plane $\piq$.
We partition the points of $\PG(2,q^3)$ using the point orbits of $\phi$.
 A point $P\in\PG(2,q^3)$ has Type $\I$, Type  $\II$ or Type  $\III$ respectively if the orbit of $P$ under $\phi$ is a point, three collinear points, or three non-collinear points.
Dually,  a line $b$ has Type $\I$, Type $\II$ or Type $\III$ respectively if the orbit of $b$ under $\phi$ is a line, three concurrent lines, or three non-concurrent lines.
Thus, a point of $\piq$ has Type $\I$; a point that lies in a  line of $\piq$ but not in $\piq$ has Type $\II$; and a point that lies on no lines of $\piq$ has Type $\III$. A line has Type $\I$, Type $\II$ or Type $\III$ respectively if it meets $\piq$ in $q+1$,  one or zero points.

We further partition the points of $\PG(2,q^3)$ with respect to a specified Type $\III$ point as follows.
Let $T$ be a Type $\III$ point and let $\SG$ be the subgroup of $\PGL(3,q^3)$ that fixes setwise the $\Fq$-plane $\piq$ and fixes setwise the line $T^\phi T^{\phi^2}$.
Then (see for example \cite[Thm 3.1]{BJ-ext1})  $\SG$ has order $q^2+q+1$ and $\SG$ fixes exactly  three points, namely the triangle  $T,T^\phi,T^{\phi^2}$. The point orbits of $\SG$ are well known.
In particular, the $q^3-1$ points of the line $ T^\phi T^{\phi^2}$ distinct from $T^\phi, T^{\phi^2}$ are partitioned by $\SG$ into $q-1$ pointsets  of size $q^2+q+1$.
These pointsets have been studied in several other contexts. In \cite{BJ-ext1}, these pointsets are called exterior splashes with carriers $T^\phi,T^{\phi^2}$, and it is shown that they are projectively equivalent to: scattered $\Fq$-linear sets of $\PG(1,q^3)$ (which are of pseudoregulus type with transversal points $T^\phi,T^{\phi^2}$); Sherk surfaces of size $q^2+q+1$;  and Bruck covers of a circle geometry $CG(3,q)$ with carriers $T^\phi,T^{\phi^2}$. For more information on scattered linear sets and psuedoregulus type see \cite{lavr10} and \cite{LMPT-2014}.
In this article, we call these $q-1$ pointsets \emph{$T$-slses} (which stands for $T$-scattered-linear-sets).

The set of all point orbits of $\SG$ is denoted $\orb(\SG)$.
So $\orb(\SG)$ is a partition of the points of $\PG(2,q^3)$ consisting of the three fixed points $T,T^\phi,T^{\phi^2}$; $q-1$ scattered linear sets on each side of the triangle $TT^\phi T^{\phi^2}$;
and $q^4-q^3-q+1$ $\Fq$-planes disjoint from the sides of  the triangle $T T^\phi T^{\phi^2}$.
This article studies the structure of the partition $\orb(\SG)$ in detail.  The motivation for this study lies in the application to the Figueroa projective plane.

The finite Figueroa planes have order $q^3$, $q$ a prime power, $q>2$. They are of particular interest as there are few  known  infinite  families of non-translation planes, these families are listed in \cite[Thm 11.1]{john} which remarks that there are three classes not connected to translation planes, namely Figueroa, Hughes and Coulter-Mathews. Hence, it is an interesting question whether the geometric construction method for a Figueroa plane can be generalised in some way. 
Grundh\"ofer \cite{grundhofer}  gives the following  synthetic  construction of the Figueroa plane FIG$(q^3)$ from the plane $\PG(2,q^3)$. The points of FIG$(q^3)$ are the points of $\PG(2,q^3)$. The lines of FIG$(q^3)$ consist of the lines of $\PG(2,q^3)$ of Types $\I$ and $\II$, and a third class of lines we call  \emph{Fig-blocks}. To construct the Fig-blocks, we first define an involutory bijection $\mu$ between points  of Type \III\ and lines of Type \III:
let $P$ be a Type-\III-point and $\ell$ a Type-\III-line, and define
$$
\begin{array}{rrcl}
\mu\colon &P &\mapsto& P^\phi P^{\phi^2},\\
& \ell& \mapsto& \ell^\phi\cap\ell^{\phi^2}.
\end{array}
$$

Let $T$ be a Type $\III$ point, then corresponding to the line $T^\mu$ is a Fig-block which we denote by $\figg{T}$, whose points are given by the following.

\begin{defn}\Label{001}  The \emph{Fig-block} $\figg T=\E_T\cup\F_T$, where
   $\E_T$ is the set of $q^2+q+1$ Type  $\II$   points on the line $T^\mu$,
and $\F_T=\{\ell^\mu \,|\, T\in\ell\}$.
\end{defn}

In this article we consider $\figg{T}$ as a set of points in $\PG(2,q^3)$ and try to describe the position its points.
Note that the Type $\III$  line $T^\mu$ and the Fig-block $\figg{T}$ have $q^2+q+3$ points in common, namely  $\figg{T} \cap T^\mu=\E_T\cup\{T^\phi, T^{\phi^2}\}$.

Brown~\cite{julia} proved the following relationship between certain elements of  $\orb(\SG)$ and the Fig-block $\figg{T}$.
 Let $T$ be a Type $\III$ point, then
the Type $\III$ points in the Fig-block $\figg{T}$ distinct from $T^\phi$, $T^{\phi^2}$ are precisely  the points on a certain set of $q-2$ $\Fq$-planes in $\orb(\SG)$; these (together with $\piq$) are called   \emph{$T$-planes} and are defined in Result~\ref{59}.

A central theme of this article is to investigate the structure of the partition $\orb(\SG)$, and in particular, the structure of $T$-planes within $\orb(\SG)$ with the aim of better understanding the Figueroa plane construction.
The insight gained in this study was used by the authors in \cite{BHJ-fig2} to demonstrate a new geometric description of the Figueroa plane, and  is used here to give a
geometric characterisation of Fig-blocks in Theorem~\ref{16}.

The work in this article proceeds as follows. Section~\ref{sec2} introduces the coordinates used.
In Theorem~\ref{111} we categorise the elements of $\orb(\SG)$ into seven classes, and in Lemma~\ref{135} we show that $\phi$ permutes elements of $\orb(\SG)$ within these seven classes.
Section~\ref{309} looks at the action of the map $\mu$ on elements of $\orb(\SG)$. One motivation in this study is to understand better the properties of $\mu$ to see if there are avenues for generalising the Figueroa construction.
As $\mu$ is used to construct the Fig-block $\figg{T}$, $\mu$ maps the Type \III\ lines through $T$ to the points in the so called $T$-planes, and we show that this mapping preserves elements of $\orb(\SG)$. Further, we look at the action of $\mu$ on the lines of a $T$-plane. It is perhaps interesting in the linear set setting  to note that $\mu$ maps the lines of a $T$-plane to a scattered linear set of the line $m_T$; and maps the points of a $T$-plane to a pencil with vertex $T$ and base a scattered linear set of $m_T$. That is, $\mu$ is another map (different from projection and splash) which maps certain $\Fq$-planes to scattered linear sets.

Section~\ref{389} looks at two well known maps from $\Fq$-planes to scattered linear sets, namely  the
projection from $T$ onto  $T^\phi  T^{\phi^2}$; and the splash onto  $T^\phi  T^{\phi^2}$. The motivation for looking at these is in studying the position of the points in the Fig-block $\figg{T}$. Considering the projection maps leads to  the following geometric visualisation of the partition $\orb(\SG)$.
The lines through $T$, distinct from   $T T^\phi $, $ T T^{\phi^2}$, can be partitioned into $q-1$ pencils of size $q^2+q+1$. Each  pencil meets $T^\phi  T^{\phi^2}$ in an element of $\orb(\SG)$. Further, if $\mathcal B$ is an $\Fq$-plane in $\orb(\SG)$, then the lines $\{TX, X \textup{ a point of }{\mathcal B}\}$ form one of these pencils. Section~\ref{sec6} focusses on the geometric relationship of the partition to the Figueroa plane, and determines the distribution of $T$-planes among these pencils, which is illustrated in Figure~\ref{fig3}.
Section~\ref{311} looks at the projection of $T$-planes from vertices other than $T$ onto the line $T^\phi T^{\phi^2}$. This is used to conclude with a geometric characterisation of the points in the Fig-block $\figg{T}$.

\section{Coordinates}\Label{sec2}

Let $P$ be a point of $\PG(2,q^3)$ and let $\vec P=(x,y,z)$ denote the homogeneous coordinates of $P$, so $x,y,z\in\Fqqq$, not all zero. Further, for all $\lambda\in\Fqqqstar$, the vectors $\lambda (x,y,z)$ are homogeneous coordinates for the same point. Similarly, a line $\ell$ has homogeneous coordinates denoted by $\vec \ell=[a,b,c]\equiv \lambda [a,b,c]$, $a,b,c\in\Fqqq$, not all zero, $\lambda\in\Fqqqstar$. A point $P$ lies on a line $\ell$ iff $\vec P\cdot\vec \ell=0$. Let  $\Norm$ denote the Norm function $\Norm \colon\Fqqq\rightarrow \Fq$, where for $x\in\Fqqq$, we have $$\Norm (x)=x^{q^2+q+1}.$$

In this article, we assume the collineation $\phi\in\PGammaL(3,q^3)$ is:
$$\begin{array}{rrcl}
\phi\colon\  & (x,y,z)& \longmapsto& (z^q,x^q,y^q)\\
& [d,e,f] & \longmapsto & [f^q,d^q,e^q]
\end{array}.$$
We can do this without loss of generality 
as any planar order 3 collineation of  $\PG(2,q^3)$
is conjugate to $\phi$ or $\phi^2$ by \cite[Theorem 2]{BJ}.
Moreover the definitions introduced in Section \ref{388} are unaffected
if $\phi$ and its inverse $\phi^2$ are interchanged.

The fixed points of $\phi$ are exactly those on the
$\Fq$-plane $\piq=\{(x,x^q,x^{q^2})\mid x\in\Fqqqstar\}.$
Let $P$ be a point and $\ell$ a line with coordinates $\vec P=(x,y,z)$ and $\vec\ell=[d,e,f]$.  To determine the Type of $P$ and $\ell$, we compute the ranks of the following matrices
\begin{eqnarray}\Labele{e2}
A_P=\begin{pmatrix}
 x&y&z\\z^q&x^q&y^q\\y^{q^2}&z^{q^2}&x^{q^2}\\
 \end{pmatrix},\qquad\qquad
A_\ell
 = \begin{pmatrix}
d&f^q&e^{q^2}\\
e&d^q&f^{q^2}\\
f&e^q&d^{q^2}\\
 \end{pmatrix}.
\end{eqnarray}
The point $P$  has Type $\I$, $\II$ or $\III$ if the   matrix  $A_P$
 has rank $1$, $2$ or $3$ respectively; and the line
  $\ell$
  has Type $\I$, $\II$ or  $\III$ if the  matrix  $A_\ell$
 has rank $1$, $2$ or $3$ respectively.

The group of collineations  of $\PG(2,q^3)$  setwise fixing $\piq$ is transitive on Type $\II$ points and  is transitive on Type $\III$ points, see \cite[Theorem B(b)]{Demp2}. We specify  a Type $\III$ point $T$, and throughout this article, without loss of generality,  we let  $T$ have coordinates $$\vec T=(0,0,1),$$
so $\vec T^\phi=(1,0,0)$ and $\vec T^{\phi^2}=(0,1,0)$.
The work with the line $m_T=m^\mu=T^\phi T^{\phi^2}$ which has coordinates $$\vec m_T=[0,0,1].$$ The points of $m_T$ (distinct from $T^\phi, T^{\phi^2}$) have coordinates $\vec P_s=(s,1,0)$, $s\in\Fqqqstar$.
By checking the determinant of the $3\times3$ matrix whose rows are points in the orbit of $P_s$ under $\phi$, the point  $P_s\in m_T$ has Type $\II$ if $\Norm(s)=-1$, and Type $\III$ if $\Norm(s)\neq -1$.
A line $\ell$ containing $T$ (distinct from $T T^\phi$, $T T^{\phi^2}$) has coordinates $\vec\ell=[r,1,0]$ for some $r\in\Fqqqstar$, and $\ell$ has Type $\II$ if $\Norm(r)=-1$, and Type $\III$ if $\Norm(r)\neq-1$.

Let $\SG$ be the subgroup of $\PG(2,q^3)$ that fixes setwise $\piq$ and fixes setwise $m_T$,
so  $\SG=\{\psi_t \mid t \in\Fqqqstar\}$, where $\psi_t \in\PGL(3,q^3)$,
$t \in\Fqqqstar$, acts on a point $X\in \PG(2,q^3)$ with coordinates $\vec X=(x,y,z)$ as follows
\begin{eqnarray}\Labele{eqnpsi}
X^{\psi_t} &=&
\begin{pmatrix} t &0&0\\0&t ^q&0\\0&0&t ^{q^2}
\end{pmatrix}
\begin{pmatrix}x\\y\\z\end{pmatrix}.
\end{eqnarray}
The  collineation $\psi_t$, $t\in\Fqqqstar$ fixes the three points $T,T^\phi,T^{\phi^2}$. For a point  $P$ of $\PG(2,q^3)$ where $P\notin\{T,T^\phi,T^{\phi^2}\}$, let $P^{\SG}$ denote the point orbit of $P$ under $\SG$. That is, if $\vec{P}=(x,y,z)$,  for some $x,y,z\in\Fqqq$, then
\[
P^{\SG}=\{(t x,t ^qy,t ^{q^2}z)\mid t\in\Fqqqstar\}.
\]
We denote the set of point orbits of $\SG$ by
 $$\orb(\SG)=\{P^\SG\,|\, P\in\PG(2,q^3)\}.$$
 It is well known that $\SG$ fixes exactly three points, namely $T, T^\phi, T^{\phi^2}$. The  remaining elements of $\orb(\SG)$ are $\Fq$-planes that are disjoint from the sides of the  triangle $T T^\phi T^{\phi^2}$, and scattered linear sets on each side of  the  triangle $T T^\phi T^{\phi^2}$.
 The group $\SG$ acts regularly on each element of $\orb(\SG)$, and $\piq$ is a point orbit of $\SG$. So each point orbit of $\SG$ consists of points which all have the same type. Moreover the $\Fq$-planes in $\orb(\SG)$ consist of  lines which all have the same type.
 The next result uses this fact to categorise and count the different elements of $\orb(\SG)$.
The $q-1$ point orbits of $\SG$ of size $q^2+q+1$ on the side $T^\phi T^{\phi^2}$ (respectively $T^{\phi ^2}T$ or $TT^{\phi}$) are called  \emph{$T$-slses} (respectively \emph{$T^\phi$-slses} or \emph{$T^{\phi^2}$-slses}).
We call a $T$-sls a \emph{Type-X-$T$-sls} if it contains only Type X points for some  $X\in\{\II,\III\}$.

\begin{theorem}\Label{111} The elements of $\orb(\SG)$ consist of the following:

\begin{enumerate}
\item\label{111z} $\{T\}$, $\{T^\phi\}$, and $\{T^{\phi^2}\}$,
\item\label{111a1} one Type-\II-$T$-sls, one   Type-\II-$T^\phi$-sls and one  Type-\II-$T^{\phi^2}$-sls,
\item\label{111a2} $q-2$  Type-\III-$T$-slses, $q-2$  Type-\III-$T^\phi$-slses, and $q-2$ Type-\III-$T^{\phi^2}$-slses,
\item\label{111b} $\piq$ (which contains only Type $\I$ points and only Type $\I$ lines),

\item\label{111c} $q^3-q-3$ $\Fq$-planes which consist of only Type $\II$ points, and only Type $\III$ lines,
\item\label{111d} $q^3-q-3$ $\Fq$-planes which consist of only Type $\III$ points, and only Type $\II$ lines,

\item\label{111e} $q^4-3q^3+q+6$ $\Fq$-planes that consist of only Type $\III$ points and only Type $\III$ lines.
\end{enumerate}
\end{theorem}

\begin{proof}
This result follows from computations given in Brown \cite{julia} and we complete the details here. Let $P$ be a point with coordinates $\vec{P}=(x,y,z)$.
If exactly two of $x,y,z$ are equal to zero, then $P$ is $T,T^\phi$ or $T^{\phi^2}$.
If exactly one of $x,y,z$ is equal to zero, then $P$ is on a side of the triangle $T T^\phi T^{\phi^2}$, so  $P^{\SG}$ is a set of collinear points  lying on a side of the triangle $T T^\phi T^{\phi^2}$, and hence is a $T$-, $T^\phi$- or $T^{\phi^2}$-sls. Parts \ref{111a1} and \ref{111a2} follow as points in an element of $\orb(\SG)$ have a single type, and a Type $\III$ line contains exactly $q^2+q+1$ Type $\II$ points.

Now suppose  $xyz\neq0$, then $P^{\SG}$ is an $\Fq$-plane. The lines of this $\Fq$-plane are the lines $\ell^{\SG}$ where $\vec\ell=[yz,zx,xy]$.
Let $X=x^{1+q}-y z^q,Y=y^{1+q}-z x^q,Z=z^{1+q}-x y^q$.
Recalling the notation from (\ref{e2}), the key equation we need is on   \cite[page 37]{julia}, namely
\begin{eqnarray}\Labele{e11}
\Norm (X)-\Norm(x)\det A_P=\Norm (Y)-\Norm(y)\det A_P=\Norm (Z)-\Norm(z)\det A_P=-  \det A_\ell.
\end{eqnarray}
If the point $P$ has Type $\I$, then $P^\SG=\piq$, and we have part \ref{111b}.
Suppose  the point $P$ has Type $\II$, then $\textup{rank}(A_P)=2$ and so $\det A_P=0$. Further, as $xyz\neq0$, we have
 $XYZ\neq0$. So (\ref{e11}) becomes
 $\Norm (X)=\Norm (Y)=\Norm (Z)=-\det A_\ell$. As $XYZ\neq0$, we have  $\Norm(X)\neq0$ and so
 $\det A_\ell\neq0$, hence $\ell$ has Type $\III$. That is, if $P$ is a Type $\II$ point not on a side of the triangle $T T^\phi T^{\phi^2}$, then the lines of the $\Fq$-plane $P^{\SG}$ are all Type $\III$. There are a total of $(q^3-q)(q^2+q+1)$ Type $\II$ points, with $3(q^2+q+1)$ lying on a side of triangle $T T^\phi T^{\phi^2}$. Hence there are $q^3-q-3$ such $\Fq$-planes, giving part \ref{111c}.   By symmetry, there are $q^3-q-3$ $\Fq$-planes  in $\orb(\SG)$ with only Type $\II$ lines, and only Type $\III$ points, giving part \ref{111d}.
It follows that the remaining $\Fq$-planes in $\orb(\SG)$ have Type $\III$ points and Type $\III$ lines, giving part \ref{111e}.
 \end{proof}

 Coordinates for the  $q-1$ $T$-slses follow from \cite{sherk}, and are explicitly given in \cite[Section 4]{LMPT-2014}.

   \begin{result}\Label{1000}  For $\theta\in\Fqqqstar$, the $T$-slses are given by  $$\S_\theta =\{(x\theta, x^q,0)\mid x\in\Fqqqstar \},$$ then
\begin{enumerate}
\item $\S_{\theta_1}=\S_{\theta_2}$ iff $\Norm (\theta_1)=\Norm (\theta_2)$,
 \item  $\S_\theta$ is a Type-\II-$T$-sls if $\Norm(\theta)=-1$, otherwise $\S_\theta$  is a Type-\III-$T$-sls.
\end{enumerate}
\end{result}

There is another useful set of lines related to $\orb(\SG)$.
  If $\S$ is a $T$-sls, then the set of lines  $$T\S=\{T X\,|\,X\in\S\}$$
   all have the same type.
  We call $T\S$ a
 \emph{Type-X-$T$-sls-pencil}, for some $X\in\{\II,\III\}$, if all the lines in  the pencil have Type X.
It is straightforward to compute the types of lines in the pencil $T\S_\theta$, $\theta\in\Fqqqstar$.

\begin{lemma}\Label{100} Let $\theta\in\Fqqqstar$, then
   $T\S_\theta$ is a Type-\II-$T$-sls-pencil  if $\Norm (\theta)= 1$, otherwise it  is a Type-\III-$T$-sls-pencil.
   \end{lemma}

   \begin{proof} Let $\theta\in\Fqqqstar$.
   By Result~\ref{1000}, points of $\S_\theta$ have coordinates $\vec P_r=(r\theta,r^q,0)$ for $r\in\Fqqqstar$. Hence lines of the pencil $T\S_\theta$ have coordinates $\vec \ell_r=[r^q,-r\theta,0]$. The determinant of the matrix $A_{\ell_r}$ defined in (\ref{e2}) is $\Norm(r)+\Norm(-r\theta)=\Norm(r)(1-\Norm(\theta))$. This is zero iff $\Norm(\theta)=1$. That is, $\ell_r$ is a Type $\II$ line if $\Norm(\theta)=1$, otherwise it is a Type $\III$ line.
   \end{proof}

Using Result~\ref{1000} and Lemma~\ref{100}, we summarise the relationships between the point type of $T$-slses and the line type of $T$-sls-pencils.

\begin{corollary}\Label{110} There is  one Type-\II-$T$-sls-pencil and $q-2$ Type-\III-$T$-sls-pencils.
\begin{enumerate}
\item If $q$ is even, then
\begin{itemize}
\item the  Type-\II-$T$-sls-pencil  meets the line $m_T=T^{\phi} T^{\phi^2}$ in a Type-\II-$T$-sls,
\item each Type-\III-$T$-sls-pencil  meets $m_T$ in  a Type-\III-$T$-sls.
\end{itemize}
\item If $q$ is odd, then
\begin{itemize}
\item the  Type-\II-$T$-sls-pencil  meets the line $m_T$ in a Type-\III-$T$-sls,
\item one Type-\III-$T$-sls-pencil  meets $m_T$ in  a Type-\II-$T$-sls,
\item $q-3$ Type-\III-$T$-sls-pencils  meet $m_T$ in  a Type-\III-$T$-sls.
\end{itemize}
 \end{enumerate}
 \end{corollary}

Brown~\cite{julia} showed that the Fig-block $\figg{T}$ can be constructed by replacing the $q-2$ Type-\III-$T$-slses on the line $T^\phi T^{\phi^2}$ with a certain set of $q-2$
$\Fq$-planes in $\orb(\SG)$. Coordinates for  these planes are computed in \cite{julia}, and we need the following definition of a \emph{$T$-plane} and summary of properties of $T$-planes.

\begin{result}\Label{59}
Let $\theta\in\Fqqqstar$, then the pointset  $$\Pi_\theta=\{ (r\theta^{q+1},r^q,r^{q^2}\theta)\,|\, r\in\Fqqqstar\}$$ forms an $\Fq$-plane called a \emph{$T$-plane}. Further,
\begin{enumerate}
\item  the lines of $\Pi_\theta$ have coordinates $[s,s^q\theta^{q+1},s^{q^2}\theta^q]$, $s\in\Fqqqstar$,
\item $\Pi_\theta=\Pi_\kappa$ iff $\Norm (\theta)=\Norm (\kappa)$, so there are $q-1$ distinct $T$-planes,
\item $\Pi_1=\piq$.
\item  The Type $\III$ points in the Fig-block $\figg{T}$ consist of the two points $T^\phi$, $T^{\phi^2}$,  and the points in
the $q-2$ $T$-planes $\Pi_\theta$, $\theta\in\Fqqqstar$ with  $\Norm(\theta)\neq 1.$
\end{enumerate}
\end{result}

Note that the $q-2$ $T$-planes $\Pi_\theta$ with $\Norm(\theta)\neq1$ have only Type $\III$ points and Type $\III$ lines, so are in the class of $\orb(\SG)$ that is described in Theorem~\ref{111}(\ref{111e}).

We conclude by computing the coordinates for points and lines in a general $\Fq$-plane in $\orb(\SG)$.

\begin{lemma}\Label{387}
Let $\mathcal B$ be an $\Fq$-plane in $\orb(\SG)$. Then there are $x,y,z\in\Fqqqstar$ such that: points in $\mathcal B$ have coordinates $(t x,t ^qy,t ^{q^2}z)$ for $t \in\Fqqqstar$; and lines in $\mathcal B$ have coordinates $[yzs,xzs^{q}, xys^{q^2}]$ for $s\in\Fqqqstar$.
\end{lemma}

\begin{proof}
  Let $\mathcal B$ be an $\Fq$-plane in $\orb(\SG)$, so $\mathcal B=P^{\SG}$ for some $\vec P=(x,y,z)$ with $x,y,z\in\Fqqqstar$.  Hence every point of $\mathcal B$ has coordinates   $P^{\psi_t}=(t x,t ^qy,t ^{q^2}z)$ for $t \in\Fqqqstar$.
  The lines of $\mathcal B$ join two points of $\mathcal B$. The line joining
the two points   $P^{\psi_t}$ and  $P^{\psi_v}$, $t,v\in\Fqqqstar$, $t\neq v$ has coordinates
$ [yz(t ^qv^{q^2}-t ^{q^2}v^q),xz(t ^{q^2}v-t v^{q^2}), xy(t v^q-t ^qv)]
 $. Setting $s=t ^qv^{q^2}-t ^{q^2}v^q$ and noting that $s$ ranges through $\Fqqqstar$ when $t\neq v$ ranges through $\Fqqqstar$ completes the proof.
   \end{proof}


\section{The action of $\phi$ on elements of $\orb(\SG)$}\Label{305}

In this section we look at  the action of $\phi$ on
elements of $\orb(\SG)$.

\begin{lemma}\Label{135}
The map $\phi$ permutes elements of $\orb(\SG)$, and preserves the  seven categories given in Theorem~\ref{111}.
\end{lemma}

\begin{proof} It is straightforward to
verify that for $t\in\Fqqqstar$, the collineations $\psi_t$ and $\phi$ commute,  so they preserve each others point orbits.
That is,   $\SG$ permutes the point orbits of $\phi$, and  $\phi$ permutes the point orbits of $\SG$.
It follows that $\phi$ permutes the elements of $\orb(\SG)$. As $\phi$ preserves types, $\phi$ preserves the seven categories given in  Theorem~\ref{111}.
 \end{proof}

By definition, $\piq$ is fixed pointwise by  $\phi$, so in particular it is fixed setwise.  We next look at whether $\phi$ fixes setwise any other elements of $\orb(\SG)$.

      \begin{theorem}
\Label{113}
Exactly $\gcd(3,q-1)$  $\Fq$-planes in $\orb(\SG)$ are fixed setwise by $\phi$, namely  $\pi_\lambda=\{ (x,x^q,\lambda x^{q^2})\,|\, x \in \Fqqqstar\}$, where $\lambda^3=1$, $\lambda\in\Fq$. Further,    $\pi_1=\piq$, and if $\gcd(3,q-1)=3$, then the other two  $\Fq$-planes fixed setwise by $\phi$ have Type $\III$ points and Type $\III$ lines.
\end{theorem}

\begin{proof}
We begin by discussing the solutions  to the equation $x^3=1$, $x\in\Fqqq$.
Let $\tau$ be a multiplicative generator of $\Fqqqstar$. Note that $\gcd(3,q-1)=\gcd(3,q^3-1)$.  If $\gcd(3,q-1)=1$, then $x^3=1$ has a unique solution, namely $x=1$. If $\gcd(3,q-1)=3$, then $x^3=1$ has three solutions, namely  $1,\tau^{\frac13(q^3-1)},\tau^{\frac23(q^3-1)}$, so the three solutions to $x^3=1$ lie in $\Fq$.

We first show that the $\Fq$-plane $\pi_\lambda=\{ (x,x^q,\lambda x^{q^2})\,|\, x \in \Fqqqstar\}$, where $\lambda^3=1$, $\lambda\in\Fq$ is fixed setwise by $\phi$. If $\lambda=1$, then $\pi_1=\piq$ which is fixed pointwise by $\phi$. Suppose $\gcd(3,q-1)=3$ and $\lambda\neq1$, so $\lambda$ is either $\tau^{\frac13(q^3-1)}$ or $\tau^{\frac23(q^3-1)}$.  Note that $\pi_\lambda=(1,1,\lambda)^\SG$, so $\pi_\lambda\in  \orb(\SG)$. By Lemma~\ref{135},  $\phi$ permutes elements of $\orb(\SG)$, we show that $\phi$ fixes setwise $\pi_\lambda$ by showing that $\phi$ maps the point $P$ with coordinates  $\vec P=(1,1,\lambda)$ to a point of $\pi_\lambda$. Using the coordinates for $\phi$ given in Section~\ref{sec2}, we
 compute $P^\phi=(\lambda, 1,1)$.
Consider   the case $\lambda=\tau^{\frac13(q^3-1)}$. As $\gcd(3,q-1)=3$,  $q^2+q+1\equiv 0\pmod3$. Let
   $x=\tau^{-\frac{q^2+q+1}{3}}$, then $x^{q-1}=\frac1\lambda$ and $\lambda x^{q^2-1}=\lambda (x^{q-1})^{q+1}=\lambda (\frac1\lambda)^{q+1}=\frac1\lambda$. Thus $P^\phi\equiv(1,\frac1\lambda,\frac1\lambda)=(1,x^{q-1},\lambda x^{q^2-1})$ which lies in $\pi_\lambda$. That is, we have demonstrated a point $P$ with $P,P^\phi\in\pi_\lambda$. As $\phi$ permutes elements of $\orb(\SG)$, it follows that $\pi_\lambda^\phi=\pi_\lambda$. The case  when $\lambda=\tau^{\frac23(q^3-1)}$ is similar.

Let $\mathcal B$ be an $\Fq$-plane  in $\orb(\SG)$, and suppose $\mathcal B$ is fixed setwise by $\phi$. We show that $\mathcal B$ is $\pi_\lambda$ for some $\lambda\in
\Fq$ with $\lambda^3=1$.  Let $P$ be a point of $\mathcal B$, so $P$ has coordinates $\vec P=(1,y,z)$ for some $y,z\in\Fqqqstar$, and
 $\mathcal B= P^{\SG}$.
 We have $P^\phi=(z^q,1,y^q)$. As $\mathcal B$ is fixed setwise by $\phi$, $P^\phi$ lies  in $\mathcal B$ and so  $(z^q,1,y^q)\equiv(r,r^qy,r^{q^2}z)$ for some $r\in\Fqqqstar$. So we need
  $\frac{r} {z^q}=r^qy=\frac{r^{q^2}z}{y^q}$, and rearranging gives the three equations
$$ r=r^qyz^q,\quad r^qy^{q+1}=r^{q^2}z,\quad r^{q^2}z^{q+1}=ry^q.$$ If we raise the second equation to the power of $q^2$, and the third equation to the power of $q$, we have the three equations
\begin{eqnarray}
  r&=&r^qyz^q,\Labele{D}
  \\
    ry^{q^2+1}&=&r^qz^{q^2},\Labele{F}\\
  rz^qz^{q^2}&=&r^qy^{q^2}.\Labele{E}
    \end{eqnarray}
    Computing the ratio of (\ref{D}) and (\ref{F}) and rearranging gives
    \begin{equation}\Labele{G}
    \frac1{\Norm(y)}=u \quad \textup{where}\quad u=\frac{yz^q}{y^qz^{q^2}}.
    \end{equation}
    As $\Norm(y)\in\Fq$, we have $u\in\Fq$, and so $u^q=u$, hence
    \begin{equation}\Labele{J}
    (y^qz^{q^2})^3=\Norm(y)\Norm(z).
    \end{equation}
Using $u=u^q=u^{q^2}$  when computing and rearranging
   the ratio of  (\ref{E}) and (\ref{D})   gives
      $\Norm(z)=u$. Equating with (\ref{G}) gives $\Norm(y)\Norm(z)=1$, and equating this with (\ref{J}) gives
   $(y^qz^{q^2})^3=1$. It follows from the first paragraph of this proof that $y^qz^{q^2}$ lies in $\Fq$. Hence $y^qz^{q^2}=y^{q^2}z =y z^{q}$ and so $u=1$. Hence by (\ref{G}), $\Norm(y)=1$.

   Let $\lambda=y^{q^2}z$, so we have $\lambda\in\Fq$, $\lambda^3=1$ and $$\vec P=(1,y,z)=(1,y,\frac{\lambda}{y^{q^2}})=(1,y,\frac{\lambda y^{q+1}}{\Norm(y)})=(1,y, \lambda y^{q+1}).$$
Let $\tau$ be a multiplicative generator of $\Fqqqstar$, then $\Norm(y)=1$ implies
     that $y=(\tau^{q-1})^i$ for some $i\in\{1,\ldots,q^2+q+1\}$. Hence
\[
\vec P= (1,(\tau^{q-1})^i,\lambda(\tau^{q^2-1})^i)\equiv (\tau^i,(\tau^i)^q,\lambda(\tau^i)^{q^2}),
\]
and so $P=(1,1,\lambda)^{\psi_{\tau^i}}$, where  $\psi_{\tau^i}$ is defined in (\ref{eqnpsi}). That is, $P$
  lies in the $\Fq$-plane $(1,1,\lambda)^\SG$, and so $\mathcal B=(1,1,\lambda)^\SG$.  Hence   the only elements of $\orb(\SG)$ which are fixed setwise by $\phi$ are the subplanes $\pi_\lambda=(1,1,\lambda)^{\SG}=\{ (x,x^q,\lambda x^{q^2})\,|\, x \in \Fqqqstar\}$, where $\lambda^3=1,\ \lambda\in\Fq$.

It follows from the first paragraph of this proof that if $\gcd(q-1,3)=1$, then there is exactly one solution  to $\lambda^3=1$, namely $\lambda=1$, hence in this case there is one fixed setwise $\Fq$-plane $\pi_1=\piq$ (which is in fact fixed pointwise by $\phi$).
If $\gcd(q-1,3)=3$, then there are three solutions in $\Fq$ to $\lambda^3=1$ and hence there are three $\Fq$-planes fixed setwise by $\phi$.
We have $\pi_1=\piq$; to determine the category of the other two $\Fq$-planes when $\gcd(q-1,3)=3$, recall that by Lemma~\ref{135}, the seven categories of elements of $\orb(\SG)$ described in Theorem~\ref{111} are preserved by   $\phi$. As $\phi$ has order 3,  elements of $\orb(\SG)$ lie in $\phi$-point orbits  of size 3. As $\gcd(3,q-1)=3$, we have $q\equiv 1\pmod3$ and the number of elements in  category (\ref{111e}) of  Theorem~\ref{111} is equivalent to $2$ modulo $3$. Hence the two $\Fq$-planes fixed setwise by $\phi$ and distinct from $\piq$ lie in this category, so have Type $\III$ points and Type $\III$ lines.
\end{proof}

\section{The action of $\mu$ on elements of $\orb(\SG)$}\Label{309}

The construction of the Figueroa plane given by Grundh\"ofer uses the involutory bijection $\mu$ between the  Type-\III-points and  the  Type-\III-lines defined in Section~\ref{388}.
In this section we  determine the action of $\mu$ on   elements of  $\orb(\SG)$. To distinguish between the cases where $\mu$ acts on the points and on the lines, we use the following notation.
$$
\begin{array}{rclcrcl}
\mupt\colon\textup{Type-\III-points}&\longrightarrow&\textup{Type-\III-lines}&\qquad&\muline\colon\textup{Type-\III-lines}&\longrightarrow&\textup{Type-\III-points}\\
P&\longmapsto& P^\phi P^{\phi^2}&&\ell&\longmapsto&\ell^\phi\cap\ell^{\phi^2}\\
\end{array}
$$
More generally, if $\mathcal B$ is an $\Fq$-plane with only Type $\III$ points, denote $\mupt(\mathcal B)=\{\mupt(P)\,|\, P \textup{ a point of } \mathcal B\}$. If $\mathcal B$ is an $\Fq$-plane with only Type $\III$ lines, denote
 $\muline(\mathcal B)=\{\muline(\ell)\,|\, \ell \textup{ is a secant line of } \mathcal B\}$.

The action of $\mu$ on  $\Fq$-planes in $\orb(\SG)$ motivates naming the following categories. The $q-1$ planes $\Pi_\theta^\phi$, $\theta\in\Fqqqstar$ are called $T^\phi$-planes, and the $q-1$ planes $\Pi_\theta^{\phi^2}$, $\theta\in\Fqqqstar$ are called $T^{\phi^2}$-planes. We note that as the maps $\phi$ and $\mu$ commute, we have $\figg{T}^\phi=\figg{T^\phi}$, and so $T^\phi$-planes are also of interest as the points in the $q-2$ $T^{\phi}$-planes $\Pi_\theta^\phi$, $\theta\in\Fqqqstar$, $\Norm(\theta)\neq1$, belong to the Fig-block $\figg{T^\phi}$.

\begin{lemma}  \Label{301}
\begin{enumerate}
\item Let $\theta\in\Fqqqstar$, $\Norm(\theta)\neq 1$, then $\muline(\Pi_\theta)=
\S_{-\frac1\theta}$ and $\mupt(\Pi_\theta)=
T\S_{\frac1\theta}$.
\item Let $\theta\in\Fqqqstar$, $\Norm(\theta)\neq 1$, then $\muline(\Pi_\theta^\phi)$ is a $T^{\phi}$-sls and $\mupt(\Pi_\theta^{\phi})$ is a pencil of lines with vertex $T^{\phi}$ and base a $T^{\phi}$-sls.
\item Let $\theta\in\Fqqqstar$, $\Norm(\theta)\neq 1$, then $\muline(\Pi_\theta^{\phi^2})$ is a $T^{\phi^2}$-sls and $\mupt(\Pi_\theta^{\phi^2})$ is a pencil of lines with vertex $T^{\phi^2}$ and base a $T^{\phi^2}$-sls.
\item Let $\mathcal B$ be an $\Fq$-plane in $\orb(\SG)$ with all Type-\III-lines that is not a $T$-plane or a $T^\phi$-plane or a $T^{\phi^2}$-plane. Then $\muline(\mathcal B)$ is the pointset of an $\Fq$-plane in $\orb(\SG)$.
\item Let $\mathcal B$ be an $\Fq$-plane in $\orb(\SG)$ with all Type-\III-points that is not a $T$-plane or a $T^\phi$-plane or a $T^{\phi^2}$-plane. Then $\mupt(\mathcal B)$ is the lineset of an $\Fq$-plane in $\orb(\SG)$.

\end{enumerate}
\end{lemma}

\begin{proof}
By Result~\ref{59}, points of $\Pi_\theta$ have coordinates $\vec P_r= (r \theta^{q+1},r^q,r^{q^2}\theta)$, $r\in\Fqqqstar$ and lines have coordinates $\vec \ell_s= [s,s^q\theta^{q+1},s^{q^2}\theta^q]$, $s\in\Fqqqstar$.
We compute
$\vec \ell_s^\phi= [s\theta^{q^2},s^q,s^{q^2}\theta^{q^2+q}]$ and
$\vec \ell_s^{\phi^2}= [s\theta^{q^2+1},s^q\theta,s^{q^2}]$. Hence the point $\muline(\ell_s)=\ell_s^\phi\cap \ell_s^{\phi^2}$ has coordinates
$(s^qs^{q^2}-s^{q^2}\theta^{q^2+q}s^q\theta,
s^{q^2}\theta^{q^2+q}s\theta^{q^2+1}-s\theta^{q^2}s^{q^2},
0)$. Dividing through by the common factor $s^{q^2}(1-\theta^{q^2+q+1})$ and rearranging gives
$ \muline(\ell_s)=(-\frac 1{\theta^{q^2}}\frac1s,\frac1{s^q},0)$. Thus $\muline(\Pi_\theta)=\{
\muline(\ell_s)\,|\,s\in\Fqqqstar\}=\{(-\frac 1{\theta^{q^2}}\frac1s,\frac1{s^q},0)\,|\,s\in\Fqqqstar\}$.
Hence by Result~\ref{1000}, $\muline(\Pi_\theta)=\S_{\kappa}$ where $\kappa=-\frac 1{\theta^{q^2}}$. As $\Norm(\theta^{q^2})=\Norm(\theta^q)=\Norm(\theta)$, by Result~\ref{1000}, $\muline(\Pi_\theta)=\S_{-\frac1{\theta}}$.

Next we show that $\mu$ maps the pointset of a $T$-plane $\Norm(\theta)\neq 1$ to the lineset of a $T$-sls-pencil.
Similar computations give      $\mupt(P_r)=[-r^q \theta^{q^2},\ r,\ 0]$.
We now compute  $\mupt(P_r)\cap m_T=(r,\ r^q\theta^{q^2},\ 0)\equiv (\frac{1}{\theta^{q^2}}r,\ r^q,\ 0)$.
Hence by Result~\ref{1000}, $\mupt(\Pi_\theta)\cap m_T=\S_{\kappa}$ where $\kappa=\frac 1{\theta^{q^2}}$. As $\Norm(\theta^{q^2})=\Norm(\theta)$, by Result~\ref{1000}, $\mupt(\Pi_\theta)\cap m_T=\S_{\frac1{\theta}}$.
Hence we have  $\mupt(\Pi_\theta)=T\S_{\frac1\theta}$. This completes the proof of part 1.
Similar computations prove parts 2 and 3.

   Let $\mathcal B$ be an $\Fq$-plane in $\orb(\SG)$. By Lemma~\ref{387},
 $\mathcal B=P^{\SG}$ for some $\vec P=(x,y,z)$ with $x,y,z\in\Fqqqstar$.
    Further,  points in $\mathcal B$ have coordinates $\vec P_r= (r x,r ^qy,r ^{q^2}z)$ for $r \in\Fqqqstar$ and lines of $\mathcal B$ have coordinates $\ell_s=[yzs,xzs^{q}, xys^{q^2}]$ for $s\in\Fqqqstar$.
 Suppose $\mathcal B$ has  all Type-\III-lines. Computing
the coordinates of $\muline(\ell_s)=\ell_s^\phi \cap \ell_s^{\phi^2}$ gives
$$s^qs^{q^2} (y^qy^{q^2}z^qz^{q^2}-x^qx^{q^2}y^{q^2}z^{q}),
 s^{q^2+1} (x^qx^{q^2}z^qz^{q^2}-x^qy^{q}y^{q^2}z^{q^2}),
s^{q+1} (x^qx^{q^2}y^qy^{q^2}-x^{q^2}y^{q}z^{q}z^{q^2})).$$
Thus $\muline(\ell_s)=Q^{\psi_t}$ where $t = s^qs^{q^2}$, $\psi_t$ is defined in  (\ref{eqnpsi}) and
the point  $Q$ has coordinates $$\vec Q=(y^qy^{q^2}z^qz^{q^2}-x^qx^{q^2}y^{q^2}z^{q}, \
x^qx^{q^2}z^qz^{q^2}-x^qy^{q}y^{q^2}z^{q^2}, \
 x^qx^{q^2}y^qy^{q^2}-x^{q^2}y^{q}z^{q}z^{q^2}).$$
 Hence $\muline(\mathcal B) =Q^{\SG}$, and so is an element of $\orb(\SG)$. By Theorem~\ref{111}, it is either a $T$-sls, a $T^\phi$-sls, a $T^{\phi^2}$-sls, or an $\Fq$-plane.
 As $\mu$ is an involution, by Lemma~\ref{301},
 $\muline(\mathcal B)$ is  a $T$-sls iff $\mathcal B$ is a $T$-plane with $\Norm(\theta)\neq 1$. Moreover, if $\mathcal B$ is not a $T^\phi$-plane or a $T^{\phi^2}$-plane, we have $\muline(\mathcal B)$ is not a $T^\phi$-sls or a $T^{\phi^2}$-sls. Hence
if $\mathcal B$  is not a $T$-plane or a $T^\phi$-plane or a $T^{\phi^2}$-plane, then
 $\muline(\mathcal B)$ is an $\Fq$-plane in $\orb(\SG)$, proving part 4.

 Suppose $\mathcal B$ has  all Type \III\ points.      We compute
    $$\mupt(P_r)= P_r^\phi  P_r^{\phi^2}=[r^qr^{q^2}(y^qz^{q^2}-x^qx^{q^2}),\ r r^{q^2}(x^{q^2}z^q-y^qy^{q^2}),\ rr^{q}(x^qy^{q^2}-z^qz^{q^2})].$$
Thus $\mupt(P_r)=\ell^{\psi_{u}}$ where $u=r^qr^{q^2}$ and
  line $\ell$ has coordinates
  \begin{equation*}
  \vec\ell=[y^qz^{q^2}-x^qx^{q^2},\ x^{q^2}z^q-y^qy^{q^2},\ x^qy^{q^2}-z^qz^{q^2}].
  \end{equation*}
 Hence
  $\mupt(\mathcal B)=\ell^{\SG}$ and so is an element of $\orb(\SG)$.
Part 5 follows similarly.
\end{proof}

Next we look at whether the map $\mu$ fixes setwise any $\Fq$-planes of $\orb(\SG)$. That is, we determine whether  there is an $\Fq$-plane $\mathcal B$ in $\orb(\SG)$ with all Type \III\ points and all Type \III\ lines such that $\mupt(\mathcal B)$ equals the set of lines of $\mathcal B$, and $\muline(\mathcal B)$  equals the set of points of $\mathcal B$.
The next result shows that the elements of $\orb(\SG)$ fixed setwise by $\mu$ are (by Theorem~\ref{113}) also fixed setwise by $\phi$.

 \begin{theorem}\Label{137}
\begin{enumerate}
\item  If $\gcd(3,q-1)=1$, then $\mu$ fixes setwise no $\Fq$-planes of $\orb(\SG)$.
\item If $\gcd(3,q-1)=3$, then $\mu$ fixes setwise exactly two  $\Fq$-planes  of $\orb(\SG)$, namely   $\pi_\lambda=\{ (x,x^q,\lambda x^{q^2})\,|\, x \in \Fqqqstar\}$, where $\lambda^3=1$, $\lambda\in\Fq$, $\lambda\neq1$.
\end{enumerate}
\end{theorem}

\begin{proof}
We first show that the $\Fq$-plane $\pi_\lambda=\{ (x,x^q,\lambda x^{q^2})\,|\, x \in \Fqqqstar\}$, where $\lambda^3=1$, $\lambda\in\Fq$, $\lambda\neq1$ is fixed setwise by $\mu$. Note that by Theorem~\ref{113}, as $\lambda\neq1$, $\pi_\lambda$ has Type $\III$ points and Type $\III$ lines, so $\mu$ does act on the points and the lines of $\pi_\lambda$. We have $\pi_\lambda=(1,1,\lambda)^\SG$, so $\pi_\lambda\in\orb(\SG)$. By  Lemma~\ref{301}, $\mu$ maps the points of $\pi_\lambda$ to the lines of an $\Fq$-plane of $\orb(\SG)$. We show that $\mu$ fixes setwise $\pi_\lambda$ by showing that $\mu$ maps the point $P\in\pi_\lambda$ with coordinates  $\vec P=(1,1,\lambda)$ to a line of $\pi_\lambda$. We use $\lambda^q=\lambda$ and
 compute $\mupt(P)=[\lambda-1,\lambda-1,1-\lambda^{2}]\equiv[1,1,-1-\lambda]$.
 As $\lambda^3-1=(\lambda-1)(\lambda^2+\lambda+1)=0$ and $\lambda-1\neq0$, we have  $\lambda^2+\lambda+1=0$ and so $-1-\lambda=\lambda^2=\frac{\lambda^3}{\lambda}=\frac{1}{\lambda}$. Hence $\mupt(P)\equiv[1,1,\frac{1}{\lambda}]\equiv[\lambda,\lambda,1]$. By Lemma~\ref{387}, lines of $\pi_\lambda$ have coordinates $[\lambda s,\lambda s^q,s^{q^2}]$ for $s\in\Fqqqstar$, so $\mupt(P)$ is a line of $\pi_\lambda$ as required.

Let $\mathcal B$ be an $\Fq$-plane in $\orb(\SG)$ with Type $\III$ points and Type $\III$ lines, and suppose $\mathcal B$ is fixed setwise by $\mu$. We will show that $\mathcal B=\pi_\lambda$ for some $\lambda\in\Fq$ with $\lambda^3=1$,  $\lambda\neq1$.
Let $P$ be a point of $\mathcal B$, so $P$ has coordinates $\vec P=(1,y,z)$ for some $y,z\in\Fqqqstar$, and $\mathcal B=P^{\SG}$. Similarly $\ell=\mupt(P)=P^\phi P^{\phi^2}$ has coordinates
 $ \vec\ell=[y^qz^{q^2}-1,\  z^q-y^qy^{q^2},\  y^{q^2}-z^qz^{q^2}]$.
 Let $\ttl = y^qz^{q^2}$, then we can write
  \begin{equation}\Labele{K}
  \vec\ell \equiv[yz(\ttl-1),\  z(\ttl^{q^2}-\Norm(y)),\  y(\ttl^q-\Norm(z))].
  \end{equation}
  As $P$ has Type $\III$, it does not lie on the line $\ell=P^\phi P^{\phi^2}$, so we have $\vec P\cdot \vec\ell\neq0$; further $yz\neq0$, and so we have
  \begin{equation}\Labele{P}
  \ttl-1+\ttl^{q^2}-\Norm(y)+\ttl^q-\Norm(z)\neq0.
  \end{equation}
  As $\mathcal B$ is fixed setwise by the map $\mu$, $\ell=\mupt(P)$ is a line of $\mathcal B$. Hence
by Lemma~\ref{387}, $\vec\ell=[yzs,zs^q,ys^{q^2}]$ for some $s\in\Fqqqstar$. Equating these with the coordinates for $\ell$ given in   (\ref{K}) gives $s=a(\ttl-1)$, $s^q=a(\ttl^{q^2}-\Norm(y))$ and $s^{q^2}=a(\ttl^q-\Norm(z))$  for some $a\in\Fqqqstar$. Raising each of these three equations to the power of $q$ and $q^2$, and using $\Norm(x)^q=\Norm(x)$ for all $x\in\Fqqq$, gives the following equations
\begin{alignat}{8}
\label{L1} s&&=&a(\ttl-1)&=&a^{q^2}(\ttl^q-\Norm(y))&&=&&a^{q}(\ttl^{q^2} -\Norm(z)),\\
\label{L2}  s^q&&=&a^q(\ttl^q-1)&=&a(\ttl^{q^2}-\Norm(y))&&=&&a^{q^2}(\ttl-\Norm(z)),\\
\label{L3}  s^{q^2}&&=&a^{q^2}(\ttl^{q^2}-1)&=&a^q(\ttl-\Norm(y))&&=&&a(\ttl^{q}-\Norm(z)).
\end{alignat}
Using (\ref{L2})  and (\ref{L3}) we have
 $$s^qs^{q^2}=a^{q^2}(\ttl-\Norm(z))\cdot a^q(\ttl-\Norm(y))=a^q(\ttl^q-1)\cdot a^{q^2}(\ttl^{q^2}-1),$$
that is,  $s^qs^{q^2}/a^qa^{q^2}=(\ttl-\Norm(z))(\ttl-\Norm(y))=(\ttl^q-1)(\ttl^{q^2}-1)$, and so
\begin{equation}\Labele{N}
\ttl^2-\ttl \Norm(y)-\ttl\Norm(z)+\Norm(y)\Norm(z)=\ttl^q\ttl^{q^2}-\ttl^{q^2}-\ttl^{q}-1.
\end{equation}
Raising (\ref{N}) to the power of $q$ and subtracting the result from (\ref{N}) gives
\begin{equation*}
(\ttl^2-\ttl^{2q })-(\ttl-\ttl^q)(\Norm(y)+\Norm(z))=\ttl^q\ttl^{q^2}-\ttl^{q^2}-\ttl^{q}-\ttl\ttl^{q^2}+\ttl+\ttl^{q^2}.
\end{equation*}
Rearranging gives
 $(\ttl-\ttl^q)(\ttl+\ttl^q+\ttl^{q^2}-1-\Norm(y)-\Norm(z))=0.$
By (\ref{P}), the second factor is not zero, hence $\ttl=\ttl^q$ and so $\ttl\in\Fq$. Using this with  (\ref{L1}) and  (\ref{L2})  gives
\begin{align}
\Labele{M1}  ss^q/aa^{q^2}=(\ttl-1)(\ttl-\Norm(z))&=(\ttl-\Norm(y))^2,\\
\Labele{M2} ss^q/aa^{q}=(\ttl-\Norm(z))(\ttl-\Norm(y))&=(\ttl-1)^2,\\
\nonumber ss^{q}/a^qa^{q^2}=(\ttl-1)(\ttl-\Norm(y))&=(\ttl-\Norm(z))^2.
\end{align}
If $\ttl=1$, then (\ref{M1}) implies
$s=0$, a contradiction. Hence $\ttl\neq1$ and similarly
 $\Norm(y)\neq\ttl$ and  $\Norm(z)\neq \ttl$. Substituting (\ref{M2}) into (\ref{M1}) gives
\begin{equation}\Labele{T}
\left(\frac{\ttl-\Norm(y)}{\ttl-1}\right)^3=1.
\end{equation}
It follows from the first paragraph of the proof of Theorem~\ref{113} that if $\gcd(3,q-1)=1$, then there is a unique solution to (\ref{T}), namely $(\ttl-\Norm(y))/(\ttl-1)=1$; and if $\gcd(3,q-1)=3$, then there are three solutions to (\ref{T}),  namely $(\ttl-\Norm(y))/(\ttl-1)=1,v,v^2$ for some $v\in\Fqqqstar$ with $v^3=1$.
We use a proof by contradiction to show that the latter  two solutions cannot occur. Suppose
\begin{equation}\Labele{V}
\frac{\ttl-\Norm(y)}{\ttl-1}=v\neq1,\quad \textup{and so } \Norm(y)=\ttl-(\ttl-1)v.
\end{equation}
Squaring the first part of (\ref{V}) and using (\ref{M1}) gives
\begin{equation}\Labele{W}
\frac{\ttl-\Norm(z)}{\ttl-1}=v^2,\quad \textup{and so } \Norm(z)=\ttl-(\ttl-1)v^2.
\end{equation}
As $\ttl=\ttl^q$, we  compute
\begin{equation}\Labele{U}
\ttl^3=\ttl \ttl^q \ttl^{q^2}=(y^q z^{q^2})(y^{q^2} z)(y z^{q})=\Norm(y)\Norm(z).
\end{equation}
Using the second parts of (\ref{V}) and (\ref{W}), this gives
\begin{equation} \Labele{R}
\ttl^3 =  (\ttl-(\ttl-1)v)(\ttl-(\ttl-1)v^2).
\end{equation}
As $1,v,v^2$ are solutions to the equation $x^3-1=0$, we have $(x-1)(x-v)(x-v^2)=x^3-1$ and so $1+v+v^2=0$ and $ v^3=1$. Substituting into (\ref{R}) and rearranging gives $(\ttl-1)^3=0$. Hence $\ttl=1$, a contradiction.
We conclude that the only solution to (\ref{T}) is $(\ttl-\Norm(y))/(\ttl-1)=1$, and
so $\Norm(y)=1$. A similar argument shows that  $\Norm(z)=1$.
Substituting $\Norm(y)=\Norm(z)=1$ into (\ref{U}) gives
  $\ttl^3=1$. That is, $\vec P=(1,y,z)$ with $\ttl =y^qz^{q^2}=y^{q^2}z$, $\ttl^3=1$, $\ttl\in\Fq$, $\ttl\neq1$. We now repeat the final two paragraphs of the proof of Theorem~\ref{113}, to deduce that
  $\vec P=(1,y,z)=(1,y,\ttl y^{q+1})=(1,1,\ttl)^{\psi_{\tau^i}}$ for some $i\in\{1,\ldots,q^2+q+1\}$.
  The only difference between this case and the proof of Theorem~\ref{113} is that here we have $\ttl\neq1$.
%
\end{proof}

\section{Two more maps acting on $\Fq$-planes in  $\orb(\SG)$: projection and splash
}\Label{389}

In this section we demonstrate some additional  relationships between  $T$-slses and $\Fq$-planes in $\orb(\SG)$. Let $\mathcal B$ be an $\Fq$-plane in $\orb(\SG)$.
In Section~\ref{309} we looked at the action of the map $\mu$ on the points and the lines of $\mathcal B$. We work with two  natural maps  from  $\mathcal B$ to the points on the line
$$m_T=T^\phi T^{\phi^2},$$
 namely   the projection from $T$ and the  splash.
 $$
\begin{array}{rcl}
\Pr\colon\textup{points distinct from $T$}&\longrightarrow&\textup{points of $m_T$}\\
P&\longmapsto&T P\cap m_T
\end{array}
$$
$$
\begin{array}{rcl}
\Sp\colon\textup{lines distinct from $m_T$}&\longrightarrow&\textup{points of $m_T$}\\
\ell&\longmapsto&\ell\cap m_T
\end{array}
$$
Note that $\Pr_{T,m_T}$ is a more precise notation for the map $\Pr$; however to simplify the notation, we abbreviate this to $\Pr$ since throughout this article the map $\Pr$ always refers to projection from the point $T$ onto the line $m_T$.
 The \emph{projection} and  \emph{splash} of the $\Fq$-plane $\mathcal B$ onto the line $m_T$ from the point $T$ are respectively  denoted
 $$
 \begin{array}{rcl}
 \Pr(\mathcal B)&=&\{T P\cap m_T\,|\,P\ \textup{is a point of }\mathcal B\},\\
\Sp(\mathcal B)&=&\{\ell\cap m_T\,|\,\ell \ \textup{is a line of }\mathcal B\}.
\end{array}$$
It is shown in \cite{luna04} and \cite{BJ-ext1} respectively  that $\Pr(\mathcal B)$  and $\Sp(\mathcal B)$ are both scattered linear sets of $m_T$. If $\mathcal B\in\orb(\SG)$, then these scattered linear sets are one of the $q-1$ linear sets on $m_T$ in  $\orb(\SG)$, with coordinates computed  in the proof of the next result. In particular, we state the coordinates of the images of  $T$-planes.

 \begin{theorem}\Label{1003ps}
 Let $\theta\in\Fqqqstar$, then  $\Pr(\Pi_\theta)=\S_{\theta^2}$ and  $\Sp(\Pi_\theta)=\S_{-\theta^2}$.
\end{theorem}

 \begin{proof}
   Let $\mathcal B$ be an $\Fq$-plane in $\orb(\SG)$, so $\mathcal B=P^{\SG}$ for some $\vec P=(x,y,z)$ with $x,y,z\in\Fqqqstar$.
    By Lemma~\ref{387},  points in $\mathcal B$ have coordinates $\vec P_r= (r x,r ^qy,r ^{q^2}z)$ for $r \in\Fqqqstar$ and lines of $\mathcal B$ have coordinates $\vec\ell_s=[yzs,xzs^{q}, xys^{q^2}]$ for $s\in\Fqqqstar$.
We compute
 $\Pr(P_r)=T P_r \cap m_T=(r x,r ^qy,0)\equiv(r\frac xy ,r^q,0)$, hence
$\Pr(\mathcal B)=\{\Pr (P_r)\,|\,r\in\Fqqqstar\}= \{(r\frac xy ,r^q,0) \,|\,r\in\Fqqqstar\}$.
 This equals $\S_{\frac xy} $ by Result~\ref{1000}.
We next compute
 $\Sp(\ell_s)=\ell_s\cap m_T=  (xzs^{q},-yzs,0)\equiv(-\frac1s\frac xy , \frac1{s^q},0)$.
 Hence $\Sp(\mathcal B)=
\{\Sp(\ell_s)\,|\, s\in\Fqqqstar\}=\{(-\frac1s\frac xy , \frac1{s^q},0)\,|\, s\in\Fqqqstar\}$. This equals
 $\S_{-\frac xy}$ by Result~\ref{1000}.
Let $\theta\in\Fqqqstar$, so by Result~\ref{59},
 $\Pi_\theta=P^{\SG}$ where  $\vec P=(\theta^{q+1},1,\theta)$. By the proof of Theorem~\ref{1003ps},    $\Pr(\Pi_\theta)=\S_{{\theta^{q+1}}}$ and $\Sp(\Pi_\theta)=\S_{-\theta^{q+1}}$. As  $\Norm(\theta^{q+1})=\Norm(\theta^2)$, it follows from  Result~\ref{1000} that   $\S_{{\theta^{q+1}}}=\S_{{\theta^2}}$ and $\S_{{-\theta^{q+1}}}=\S_{{-\theta^2}}$.
   \end{proof}

Thus, if $\mathcal B$ is an  $\Fq$-plane   in $\orb(\SG)$, then each line through $T$ meets $\mathcal B$ in either zero or one points, and  the projection of $\mathcal B$ from $T$ onto $m_T$ is a $T$-sls. That is,  the pencil $T\mathcal B=\{TP\,|\,P\in\mathcal B\}$ is a $T$-sls-pencil.
As $\phi$ permutes elements of $\orb(\SG)$,  the pencil $T^\phi \mathcal B$ is a $T^\phi$-sls-pencil; and
the pencil $T^{\phi^2} \mathcal B$ is a $T^{\phi^2}$-sls-pencil.

It is of particular interest to look at  $\piq$ and the $T$-sls consisting of Type II points. The next result summarises the properties of these by collecting the results on $\S_\theta$ and $\Pi_\theta$ when  $\Norm(\theta)=\pm1$.  The result  follows by applying Results \ref{1000}, \ref{100}, Theorem \ref{301} and Theorem \ref{1003ps}.

\begin{corollary}\Label{104b}
\begin{enumerate}
\item If $q$ even, then $\S_1$  is a Type-\II-$T$-sls, $T\S_1$  is a Type-\II-$T$-sls-pencil and \\$\S_1=\Pr(\piq)=\Sp(\piq)$.
\item If $q$ odd,
\begin{enumerate}
\item
$\S_1$   is a Type-\III-$T$-sls,  $T\S_1$  is a Type-\II-$T$-sls-pencil, and \\$\S_1=\muline(\Pi_{-1})=\Pr(\piq)=\Pr(\Pi_{-1})$.
\item
$\S_{-1}$   is a Type-\II-$T$-sls,   $T\S_{-1}$  is a Type-\III-$T$-sls-pencil,  and  \\$\S_{-1}=\Sp\big(\mupt(\Pi_{-1})\big)=\Sp(\piq)=\Sp(\Pi_{-1})$.
\end{enumerate}\end{enumerate}
\end{corollary}

\section{Projecting the Fig-block $\figg{T}$}\Label{sec6}

In this section we look at   the pointset of the Fig-blocks $\figg{T}$, $\figg{T^\phi}$, $\figg{T^{\phi^2}}$ and determine their projections from $T$ onto the line $m_T$.
This gives us more information about the structure of $\figg{T}$.

\begin{theorem}\Label{15}
\begin{enumerate}
\item If $q$ is even, then $\Pr(\figg{T})=m_T$.
\item If $q$ is odd, then $\Pr(\figg{T})=\{T^\phi,T^{\phi^2}\}\cup\S_{-1}\cup\{\S_\theta\,|\,\Norm(\theta) \textup{ a nonzero square in }\Fq\}$. Further $|\Pr(\figg{T})|$ is $2+\frac{q-1}{2}(q^2+q+1)$ or $2+\frac{q+1}{2}(q^2+q+1)$ if $q\equiv1\pmod4$ or $q\equiv3\pmod4$ respectively.
\end{enumerate}
\end{theorem}

\begin{proof}
By Result~\ref{1000}, the Type $\II$ points in $\figg{T}$ are precisely the points in the $T$-sls $\S_{-1}$.
By Result~\ref{59},  the set of Type $\III$ points in $\figg{T}$ consists of the points $T^\phi, T^{\phi^2}$ together with the points in the $q-2$ $T$-planes $\Pi_\theta$, $\theta\in\Fqqqstar$ with $\Norm(\theta)\neq1$.
 It follows from   Theorem~\ref{1003ps} that
 \begin{equation}\label{AA}
 \Pr(\figg{T})=\{\S_{{\theta^2}}\,|\,\theta\in\Fqqqstar, \Norm(\theta)\neq1\}\cup\S_{-1}\cup\{T^\phi,T^{\phi^2}\}.
 \end{equation}
 We look at the set $\D=\{\S_{{\theta^2}}\,|\,\theta\in\Fqqqstar, \Norm(\theta)\neq1\}$.

 Suppose $q$ is even.
If $\Norm(y)=1$, then $\Norm(y^2)=\Norm(y)^2=1$. As $q$ is even, $f(x)=x^2$ is an automorphism of $\Fqqqstar$,  and so $
\Norm(y)=1$ iff $\Norm(y^2)=1$.
Hence we have $\D=\{\S_{{\theta}}\,|\,\theta\in\Fqqqstar, \Norm(\theta)\neq1\}.$  Hence by (\ref{AA}), $\Pr(\figg{T})=\{\S_{{\theta}}\,|\,\theta\in\Fqqqstar, \Norm(\theta)\neq1\}\cup\S_{1}\cup\{T^\phi,T^{\phi^2}\}=\{\S_{{\theta}}\,|\,\theta\in\Fqqqstar\}\cup\{T^\phi,T^{\phi^2}\}=m_T$, proving part 1.


   Suppose  $q$ is odd and let $\K_q$ denote the set of $(q-1)/2$ nonzero squares of $\Fq$.
 We rewrite $\D$ as     $$\D=\{\S_\theta\,|\,\theta\in\K_{q^3}, \Norm(\sqrt{\theta})\neq1\}.$$
 If $\Norm(\sqrt{\theta})\neq1$, then either $\Norm(\theta)\neq1$, or $\Norm(\sqrt{\theta})=-1$ (in which case  $\Norm(\theta)=1$). Hence
$$
\D=\{\S_\theta\,|\,\theta\in\K_{q^3}, \Norm(\theta)\neq1\}\cup\{\S_\theta\,|\,\theta\in\K_{q^3}, \Norm(\sqrt{\theta})=-1\}.
$$
If $\Norm(\sqrt{\theta})=-1$, then  $\Norm(\theta)=1$, so by Result~\ref{1000}, $\{\S_\theta\,|\,\theta\in\K_{q^3}, \Norm(\sqrt{\theta})=-1\}=\S_1$. Thus
\begin{equation}\Label{BB}
\D=\{\S_\theta\,|\,\theta\in\K_{q^3}, \Norm(\theta)\neq1\}\cup\S_1.
\end{equation}
We simplify this by showing that $\theta\in\K_{q^3}$ iff $\Norm(\theta)\in\K_q$.
If $\theta\in\K_{q^3}$, then $\theta=\kappa^2$, so $\Norm(\theta)=\Norm(\kappa^2)=\Norm(\kappa)^2\in\K_q$.
Conversely, note that each nonzero element of $\Fqstar$ is the image under $\Norm$ of $q^2+q+1$ elements of $\Fqqqstar$.
As   $|\K_q|=\frac12(q-1)$ and  $|\K_{q^3}|=\frac12(q-1)(q^2+q+1)$, it follows by counting that
   $\theta\in\K_{q^3}$ iff $\Norm(\theta)\in\K_q$.
So (\ref{BB}) becomes
  \begin{eqnarray*}
  \D&=&\{\S_{{\theta}}\,|\, \theta\in\Fqqqstar,\Norm(\theta)\in\K_{q}\setminus\{1\}\}\cup\S_1\\
  &=&\{\S_{{\theta}}\,|\, \theta\in\Fqqqstar,\Norm(\theta)\in\K_{q}\}.
  \end{eqnarray*}
Hence by (\ref{AA}), $\Pr(\figg{T})=\{\S_{{\theta}}\,|\, \theta\in\Fqqqstar,\Norm(\theta)\in\K_{q}\}\cup\S_{-1}\cup\{T^\phi,T^{\phi^2}\}$. We count the total number of $T$-slses in $\Pr(\figg{T})$ by noting that $-1\in\K_q$ iff $q\equiv 1\pmod4$.  The preceding paragraph shows $-1\in\K_{q^3}$ iff $\Norm(-1)\in\K_q$. So we have   $\Norm(-1)\in\K_q$ iff $q\equiv 1\pmod4$. That is,  $\S_{-1}\in \{\S_{{\theta}}\,|\, \theta\in\Fqqqstar,\Norm(\theta)\in\K_{q}\}$ iff $q\equiv 1\pmod4$. By Result~\ref{1000}, the set $\{\S_{{\theta}}\,|\, \theta\in\Fqqqstar,\Norm(\theta)\in\K_{q}\}$ contains $|\K_q|=\frac{q-1}2$ $T$-slses.  Hence if $q\equiv 1\pmod4$, then $\Pr(\figg{T})$ contains $T^\phi,T^{\phi^2}$ and exactly $\frac{q-1}{2}$ $T$-slses;
 if $q\equiv 3\pmod4$, then $\Pr(\figg{T})$ contains $T^\phi,T^{\phi^2}$ and exactly $\frac{q+1}{2}$ $T$-slses.
\end{proof}

We can use Theorem~\ref{15} to answer the question of how the $q-1$ $T$-planes are distributed among the $q-1$ $T$-sls-pencils. The structure is illustrated in Figure~\ref{fig3}, moreover, in the figure, the elements of the Fig-block $\figg{T}$ are shaded to demonstrate its positioning.
In order to succinctly describe the structure, in the next result we say a $T$-sls-pencil \emph{arches over} a $T$-plane if each of the $q^2+q+1$ lines of the pencil contains a unique point of the $T$-plane (so each point of the $T$-plane lies in a unique line of the pencil).

\begin{figure}[h]
\centering
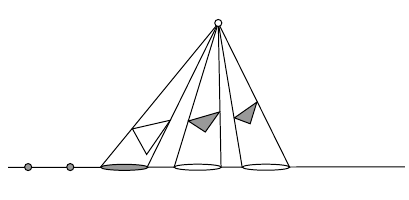\vspace*{10mm}

 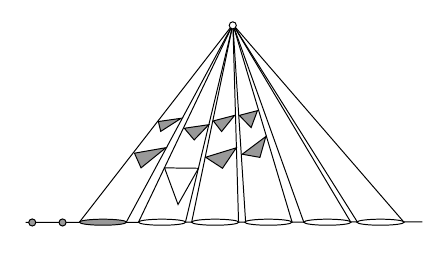\hspace*{10mm}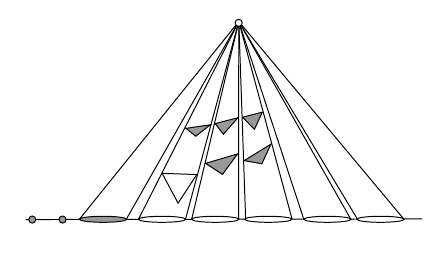\vspace*{5mm}
 \caption{Corollary~\ref{15cor} describes the position of $T$-planes in relation to the $T$-slses. In this figure the Fig-block $\figg{T}$ is shaded}
 \label{fig3}
 \end{figure}

\begin{corollary}\Label{15cor}
If $q$ is even,  each of the $q-1$ $T$-sls-pencils arches over a distinct $T$-plane.
If $q$ is odd, then the $\frac{q-1}2$ $T$-sls-pencils $T\S_{\theta}$ with $\Norm(\theta)$ a nonzero square in $\Fq$ each arch over two of the $T$-planes (so the remaining $\frac{q-1}2$ $T$-sls-pencils  arch over zero
 $T$-planes).
 \end{corollary}

\begin{proof}
Suppose $q$ is even. By the proof of Theorem~\ref{15}, the $q-2$ $T$-planes $\Pi_\theta$ with $\Norm(\theta)\neq1$ are each arched over by a distinct $T$-sls-pencil $T\S_\theta$ with $\Norm(\theta)\neq1$. Further, we have $\Pr(\E_T)=\E_T=\S_1$.
By Corollary~\ref{104b}, $\Pr(\piq)= \S_1$, so exactly one $T$-plane is arched over by the pencil $T\S_1$, namely $\piq$.  Hence each of the $q-1$ $T$-planes is arched over by a distinct $T$-sls-pencil.

If $q$ is odd, then by Corollary~\ref{104b}, $\Pr(\piq)=\Pr(\Pi_{-1})=\S_1$, and $\S_{-1}$ is not the projection of any $T$-plane. By  the proof of Theorem~\ref{15}, the $q-2$ $T$-planes forming $\F_T=\{\Pi_\theta\,|\,\Norm(\theta)\neq1\}$ project from $T$ onto the
$\frac{q-1}2$ $T$-slses in $\D=\{\S_\theta\,|\,\Norm(\theta) \textup{ a nonzero square in }\Fq\}$. Each $T$-sls in $\D$  is the projection of two $T$-planes in $\F_T$, except for $\S_1$ which is the projection of one $T$-plane, namely $\Pi_{-1}$.
As $\S_1$ is the projection of the $T$-plane $\piq$ (which is not in $\F_T$), each $T$-sls-pencil $T\S_\theta$ with $\S_\theta\in\D$ arches over exactly two $T$-planes, and the remaining $\frac{q-1}2$ $T$-sls-pencils  arch over zero
 $T$-planes.
 \end{proof}

We also describe how  the $T^\phi$-planes $\Pi_\theta^\phi$, and the $T^{\phi^2}$-planes $\Pi_\theta^{\phi^2}$  are distributed among the $T$-sls-pencils.

\begin{theorem}\Label{19}
   \begin{enumerate}
\item   $\Pr(\figg{T^\phi})=m_T\setminus\{\S_{1},T^{\phi^2}\}$.
\item
 $\Pr(\figg{T^{\phi^2}})=m_T\setminus\{\S_{1},T^{\phi}\}$.
\end{enumerate}

\end{theorem}

\begin{proof} We first show that for $\theta\in\Fqqqstar$, we have
  $\Pr(\Pi_\theta^\phi)=\Pr(\Pi_\theta^{\phi^2})=\S_{\frac1\theta}$.
Let $\theta\in\Fqqqstar$, by Result~\ref{59}, the $T$-plane $\Pi_\theta$ contains the points
$(r\theta^{q+1},r^q,r^{q^2}\theta)$ for $r\in\Fqqqstar$.
Hence
$\Pi_\theta^\phi$ contains the points
$\vec P_r=(r\theta^q,r^q\theta^{q^2+q},r^{q^2})$ for $r\in\Fqqqstar$. The line $T P_r$ has coordinates $[r^q\theta^{q^2+q}, -r\theta^q,0]$ which meets $m_T$ in the point  $(-r\theta^q,-r^q\theta^{q^2+q}, 0)$. Hence by Result~\ref{1000}, $\Pr(\Pi_\theta^\phi)=\S_\kappa$ where $\kappa=\frac{1}{\theta^{q^2}}$. As $\Norm(\frac{1}{\theta^{q^2}})=\Norm(\frac{1}{\theta})$, by Result~\ref{1000},  $\Pr(\Pi_\theta^\phi)=\S_{\frac{1}{\theta}}$.
The $\Fq$-plane $\Pi_\theta^{\phi^2}$ contains the points
$\vec Q_r=(r,r^q\theta^{q^2},r^{q^2}\theta^{q^2+1})$ for $r\in\Fqqqstar$. The line $T Q_r$ has coordinates $[r^q\theta^{q^2}, -r,0]$ which meets $m_T$ in the point  $(-r,-r^q\theta^{q^2}, 0)$. Hence by Result~\ref{1000}, $\Pr(\Pi_\theta^{\phi^2})=\S_\kappa$ where $\kappa=\frac{1}{\theta^{q^2}}$, and so $\Pr(\Pi_\theta^{\phi^2})=\S_{\frac{1}{\theta}}$.

 The Type $\II$ points of the Fig-block $\figg{T^\phi}$ lie on the line $T T^{\phi^2}$, and so each projects from $T$ onto the point $T^{\phi^2}$ of $m_T$. The Type $\III$ points of the  Fig-block $\figg{T^\phi}$ are the points $T,T^{\phi^2}$ together with the points in the $q-2$ $T^\phi$-planes $\Pi_\theta^\phi$, $\theta\in\Fqqqstar$, $\Norm(\theta)\neq1$. So $\Pr(\figg{T^\phi})$ consists of $T^{\phi^2}$ and the $T$-slses $\Pr(\Pi_\theta^\phi)$, $\theta\in\Fqqqstar$, $\Norm(\theta)\neq1$. So by the above paragraph, $\Pr(\figg{T^\phi})=\{T^{\phi^2}\}\cup\{\S_{\frac1\theta}\,|\,\theta\in\Fqqqstar, \Norm(\theta)\neq1\}$. Thus $\Pr(\figg{T})$ contains all the $T$-slses distinct from $\S_1$. As $\piq=\Pi_1^\phi$, by Theorem~\ref{1003ps}, $\Pr(\Pi_1^\phi) =\S_1$, proving part 1. Part 2 is similar.
\end{proof}

This section has looked at the image of points in the Fig-block $\figg{T}$ under the map $\Pr$. In this article we have discussed another map
 that acts on (the Type $\III$) points of $\figg{T}$, namely $\mupt$.  We conclude this section by briefly summarising how $\mupt$ can be used to map $\figg{T}$ to $m_T$.
 Denote the Type $\II$ points in $\figg{T}$ by $\E_T$ and the Type $\III$ points in $\figg{T}$ by $\F_T$, so $\figg{T}=\E_T\cup\F_T$.
Definition~\ref{001} describes how to construct the pointset of the Fig-block $\figg{T}$ from the pointset of the line $m_T$. Firstly, $\E_T$ equals the set of Type $\II$ points on $m_T$, so $\E_T\subset m_T$. Secondly, $\F_T=\{\muline(T P)\,|\, T P\ \textup{ is a Type $\III$ line}\}$. As $\mu$ is an involution, $\mupt(\F_T)$ is the set of all Type $\III$ lines through $T$.
We can further map the Type $\III$ points of $\figg{T}$ to points of $m_T$ by considering the map $\Sp\circ\mupt$.
By Corollary~\ref{104b}, $\mupt(\F_T)$ consists of the $q-2$ $T$-pencils $T\S_{\theta}$ with $\Norm(\theta)\neq1$.
Hence
we have a bijection from the Type $\III$ points of $\figg{T}$ to the points of $m_T\setminus\S_{1}$, namely
 $$\Sp(\mupt(\figg{T}\setminus\E_T))=m_T\setminus\S_{1}.$$
 By Corollary~\ref{104b}, $\S_1$ contains Type $\II$ points iff $q$ is even, that is,   the right hand side equals the set  of all Type $\III$ points of $m_T$, namely $m_T\setminus\E_T$  iff $q$ is even.

 \section{Projection vertices and Figueroa lines}\Label{311}

Consider an $\Fq$-plane $\mathcal B$ and a point $P$ not in $\mathcal B$ and not on $m_T$. Lunardon and Polverino \cite[Theorem 1]{luna04} showed that the projection of $\mathcal B$  from $P$ onto the line $m_T$ is an $\Fq$-linear set of rank $3$, so is either a club or a scattered linear set. If this linear set is a $T$-sls $\S_\theta$, then we call $P$ a \emph{projection vertex for $\mathcal B$ onto $\S_\theta$}.
Projection vertices for $\Fq$-planes are counted in  \cite{BJ-ext1}, and
we use the following results, which we have translated into the notation of this article.

\begin{result}\Label{extbig}
\begin{enumerate}
\item {\textup{\cite[Lemma 5.4]{BJ-ext1}}}\label{ext1b}
Let $\theta\in\Fqqqstar$.
 If $q$ is even, then there are either $1$ or $q^2+q+1$ projection vertices for $\mathcal B$ onto $\S_\theta$.
If $q$ is odd,  then there are either $0$, $q^2+q+1$ or $q^2+q+2$ projection vertices for $\mathcal B$ onto $\S_\theta$.
\item {\textup{\cite[Theorem 5.2]{BJ-ext1}}}\label{ext1a}
The projection of $\mathcal B$ from a point $P$ onto the line $m_T$ is equal to
$\Sp(\mathcal B)$  if and only if  $q$ is even and $P=T$.
\end{enumerate}
\end{result}

{\bfseries Remark.   } \emph{Note that the statement of   \cite[Lemma 5.4]{BJ-ext1} has an omission: the second sentence of the proof should begin with `If' instead of `As'. Hence in the $q$ odd case, the possibility $0$ can occur, and this possibility is omitted from the statement of \cite[Lemma 5.4]{BJ-ext1}.}

The next result looks at projection vertices for $T$-planes.

\begin{lemma}\Label{134}
Let $\theta,\kappa\in\Fqqqstar$ with $\Norm(\theta)\neq\Norm(\kappa)$.
The projection of $\Pi_\theta$ from any point $P\in\Pi_\kappa$ onto the line $m_T$ is the $T$-sls $\S_{-\kappa\theta}$.
\end{lemma}

\begin{proof} Let $P$ be a point of $\Pi_\kappa$, by Result~\ref{59},
$\vec P=(s \kappa^{q+1}, s^q, s^{q^2} \kappa)$ for some $s\in\Fqqqstar$. Similarly, by Result~\ref{59}, the points of $\Pi_\theta$ are $\{Q_t,|\,t\in\Fqqqstar\}$ where
$\vec Q_t=(t \theta^{q+1}, t^q, t^{q^2} \theta)$. The line $PQ_t$ has coordinates
$$[s^q t^{q^2} \theta-s^{q^2} \kappa t^q,\ s^{q^2} \kappa t \theta^{q+1}-s \kappa^{q+1}t^{q^2} \theta,\
s \kappa^{q+1}t^q-s^q t \theta^{q+1}]$$
which meets the line $m_T=[0,0,1]$ in the point $R_t$ with coordinates
 \begin{eqnarray*}
 \vec R_t
 &=&(s \kappa^{q+1}t^{q^2} \theta-s^{q^2} \kappa t \theta^{q+1},\ s^q t^{q^2} \theta-s^{q^2} \kappa t^q,\ 0)\\
 &=&(-\kappa\theta (s^qt^{q^2}\theta -s^{q^2}t^q\kappa)^q,\ s^qt^{q^2}\theta -s^{q^2}t^q\kappa,\ 0)\\
&=&(-\kappa\theta\, c,c^q,0) \qquad\qquad\textup{where } c=\frac{1}{s^qt^{q^2}\theta -s^{q^2}t^q\kappa}.\\
\end{eqnarray*}
By Result~\ref{1000}, this lies in $\S_{-\kappa\theta}$ for all values of $t$. The projection of $\Pi_\theta$ onto $m_T$ from the point $P$ is $\X=\{R_t=PQ_t\cap m_T\,|\, t\in\Fqqqstar\}$, so is a subset of $\S_{-\kappa\theta}$. We show that $\X=\S_{-\kappa\theta}$ by showing that this projection map is onto $\S_{-\kappa\theta}$.
Suppose $R_t=R_u$ for some $t,u\in\Fqqqstar$, $t\neq u$, so
 $$(-\kappa\theta (s^qt^{q^2}\theta -s^{q^2}t^q\kappa)^q,\ s^qt^{q^2}\theta -s^{q^2}t^q\kappa,\ 0)\equiv
 (-\kappa\theta (s^qu^{q^2}\theta -s^{q^2}u^q\kappa)^q,\ s^qu^{q^2}\theta -s^{q^2}u^q\kappa,\ 0).$$
 Hence
 \begin{equation}
 \Labele{E1}
 \frac{-\kappa\theta (s^qt^{q^2}\theta -s^{q^2}t^q\kappa)^q}{-\kappa\theta (s^qu^{q^2}\theta -s^{q^2}u^q\kappa)^q}=\frac{s^qt^{q^2}\theta -s^{q^2}t^q\kappa}{s^qu^{q^2}\theta -s^{q^2}u^q\kappa}.
 \end{equation}
 Writing \begin{eqnarray}\Labele{E2}
a=\frac{s^qt^{q^2}\theta -s^{q^2}t^q\kappa}{s^qu^{q^2}\theta -s^{q^2}u^q\kappa},
 \end{eqnarray}
then (\ref{E1}) says that $a^q=a$, and so $a$ is an element of $\Fq$.
Rearranging (\ref{E2}) gives
 $$\frac{\theta}{\kappa}= \frac{s^{q^2}(t^q - a u^q)} {s^{q}(t^{q^2} - a u^{q^2})}=  \frac{s^{q^2}( t - a u)^q} {s^{q}( t - a u)^{q^2}}$$
 since $a\in\Fq$.
  Taking the norms of both sides, recalling that $\Norm(x^q)=\Norm(x)$ for all $x\in\Fqqqstar$, we have $\Norm(\theta/\kappa)=1$ and so $\Norm(\theta)=\Norm(\kappa)$, a contradiction. Hence $\X= \S_{-\kappa\theta}$ as required.
\end{proof}

We next specialise to the $\Fq$-plane $\piq$, and count how many points $P$ project $\piq$   onto  one of the $T$-slses. We show how these   projection vertices  relate to the  $T$-planes in Fig-block $\figg{T}$.

\begin{lemma}\Label{res:01}
There are exactly $q^3-q^2-q-1$
points $P$ for which the projection of  $\piq$ from $P$ onto $m_T$ is a $T$-sls, namely $T$ and the points of $\figg{T}$ that are not incident with $m_T$. These  projection vertices  are distributed  as follows.
\begin{enumerate}
\item Suppose $q$ is even.
\begin{enumerate}
\item  There is exactly one projection vertex  for $\piq$ onto
 $\S_1$, namely $T$.
 \item Let $\theta\in\Fqqqstar$, $\Norm(\theta)\neq1$, there are exactly $q^2+q+1$ projection vertices for $\piq$ onto $\S_\theta$, namely the points of $\Pi_{\theta}$.
 \end{enumerate}

\item If $q$ is odd, then
\begin{enumerate}
\item  There are no projection vertices  for $\piq$ onto  $\S_{-1}$.
\item   There are $q^2+q+2$ projection vertices  for $\piq$ onto   $\S_1$, namely $T$ and the points of $\Pi_{-1}$.
\item Let $\theta\in\Fqqqstar$, $\Norm(\theta)\neq\pm1$, there are exactly $q^2+q+1$ projection vertices for $\piq$ onto $\S_\theta$, namely the points of $\Pi_{-\theta}$.
 \end{enumerate}
  \end{enumerate}
\end{lemma}

\begin{proof}
By Corollary~\ref{104b}, $\Sp(\piq)=\S_{-1}$. Hence by Result~\ref{extbig}, if $q$ is even, there is one projection vertex for $\piq$ onto $\S_1$, namely $T$; and if $q$ is odd, there are no projection vertices for $\piq$ onto $\S_{-1}$.
Let $\theta\in\Fqqqstar$, $\Norm(\theta)\neq1$ and let $P\in\Pi_{-\theta}$. By Lemma~\ref{134},  the projection of $\piq$ from $P$ onto $m_T$ is $\S_{\theta}$. This accounts for $q^2+q+1$ projection vertices for $\piq$ onto $\S_\theta$. If $q$ is even, then by Result~\ref{extbig}, there are no further projection vertices for $\piq$ onto $\S_{\theta}$.

When $q$ is odd, $\Pr(\piq)=\S_1$ by Corollary~\ref{104b}, so we have $q^2+q+2$ projection vertices for $\piq$ onto $\S_1$, namely $T$ and the points of $\Pi_{-\theta}$. By Result~\ref{extbig}, there are no further projection vertices for $\piq$ onto $\S_1$. If $\Norm(\theta)\neq \pm1$, then the $q^2+q+1$ points of $\Pi_{-\theta}$ are   projection vertices for $\piq$ onto $\S_{\theta}$. Suppose the point $Q$ was  another projection vertex for $\piq$ onto $\S_{\theta}$, then every point in $Q^\SG$ is also a projection vertex for $\piq$ onto $\S_{\theta}$. By Result~\ref{extbig}, there is at most one further projection vertex, so if $Q$ exists, then $Q^\SG=\{Q\}$, and so $Q=T$. The projection of $\piq=\Pi_1$ from $T$ onto $m_T$ is $\Pr(\Pi_1)=\S_1$ by Theorem~\ref{1003ps}. Hence if $\Norm(\theta)\neq \pm1$, the point $Q$ does not exist, so there are exactly $q^2+q+1$ projection vertices for $\piq$ onto $\S_{\theta}$.
\end{proof}

\section{Conclusion}

This analysis of the partition $\orb(\SG)$, and the structure of $T$-planes within this partition, gives insight into the Figueroa plane construction.  By Lemma~\ref{res:01}, the points of $\PG(2,q^3)$ that project $\piq$ onto a $T$-sls of the line $m_T=T^\phi T^{\phi^2}$ are precisely $T$ and the  Type III  points of the Fig-block $\figg{T}$. This gives the following  characterisation of the points in the Fig-block $\figg{T}$. 

\begin{theorem}\Label{16} Let $T$ be a Type $\III$ point.
A point $P\neq T$ lies in the Fig-block $\figg{T}$ if and only if the projection of the $\Fq$-plane $\piq$ from the point $P$ onto the line $m_T=T^\phi T^{\phi^2}$ is a $T$-sls (that is, a scattered $\Fq$-linear set of pseudoregulus type with transversal points $T^\phi$, $T^{\phi^2}$).
\end{theorem}

This setting may hold other interesting characterisations, and we propose one possible idea.
Consider the structure of $\figg{T}$ when $q$ is even. Recall that $\figg{T}=\E_T\cup\F_T$, where $\E_T\subset m_T$ is the set of Type $\II$ points in $\figg{T}$.
Applying $\phi$ to the result of Theorem~\ref{19} and using Theorem~\ref{15}, it follows  that when $q$ is even:
(i) each line through $T$ contains either one point of $\F_T$ or one point of $\E_T$; (ii) each line through $T^\phi$ contains either one point of $\F_T$ or one point of $\E_T^\phi$; (iii) each line  through $T^{\phi^2}$ contains  either one point of $\F_T$ or one point of $\E_T^{\phi^2}$.  An interesting open question is whether $\figg{T}$ is the only set of points with this structure.

\vspace*{-2mm}

{\small{
{\bfseries Author information}\\[.5mm]
S.G. Barwick. School of Mathematical Sciences, University of Adelaide, Adelaide, 5005, Australia.
susan.barwick@adelaide.edu.au\\[1mm]
A.M.W. Hui. Applied Mathematics Program, BNU-HKBU United International College, Zhuhai, China / School of Mathematical and Statistical Sciences, Clemson University, USA, huimanwa@gmail.com
\\[1mm]
W.-A. Jackson. School of Mathematical Sciences, University of Adelaide, Adelaide, 5005, Australia.
wen.jackson@adelaide.edu.au
\\[1mm]
{\bfseries Acknowledgements}
A.M.W. Hui acknowledges the support of National Natural Science Foundation of China (Grant No. 12071041).
}}

%
%
%


\begin{thebibliography}{77}


\bibitem{baker} R.D. Baker, J.M.N. Brown, G.L. Ebert and J.C. Fisher. Projective bundles. {\em Bull. Belg. Math. Soc}, {\bfseries 3} (1994) 329--336.


\bibitem{BHJ-fig2} S.G. Barwick, A.M.W. Hui  and  W.-A. Jackson. A  geometric description of the Figueroa plane. \emph{Des. Codes Cryptogr.},  {\bfseries 91} (2023) 1581--1593.

\bibitem{BJ-ext1}  S.G. Barwick and W.-A. Jackson.
Exterior splashes and linear sets of rank 3.
{\em Discrete Math.}, {\bfseries 339} (2016) 1613--1623.


\bibitem{BJ}
 L.M. Batten and  P.M. Johnson.
 The collineation groups of Figueroa planes.
 {\em Canad. Math. Bull.}, {\bfseries 36} (1993) 390--397.


\bibitem{julia}
 J.M.N. Brown. Some partitions in Figueroa planes. {\em Note di Matematica}, {\bfseries 29} (2009)  33--44.



\bibitem{grundhofer} T. Grundh\"{o}fer. A synthetic construction of the Figueroa planes. \emph{J. Geom.},  {\bfseries 26}  (1986) 191--201.


\bibitem{Demp2}
U. Dempwolff. $\textup{PSL}(3, q)$ on projective planes of order $q^3$.  {\em Geom. Dedicata}, {\bfseries 18} (1985) 101--112.





%
%
%

\bibitem{john} N.L. Johnson. Planes and processes. \emph{Discrete Math.,} {\bfseries 309} (2009) 430--461.


 \bibitem{lavr10} M. Lavrauw and G. Van de Voorde. On $\Fq$-linear sets on a projective line. \emph{ Des. Codes Cryptogr.}, {\bfseries 56} (2010) 89--104.
%

\bibitem{luna04} G. Lunardon and O. Polverino. Translation ovoids of orthogonal polar spaces. {\em Forum Math.,} {\bfseries 16} (2004) 663--669.


\bibitem{LMPT-2014} {\textup G. Lunardon, G. Marino, O. Polverino and R. Trombetti}. Maximum scattered $\Fq$-linear sets of pseudoregulus type
and the Segre variety ${\cal S}_{n,n}$. \emph{  J. Algebraic. Comb.}, {\bfseries 39} (2014) 807--831.



%
%
%
%
%
%


\bibitem{sherk} F.A. Sherk.   The geometry of GF$(q^3 )$. \emph{Canad. J. Math.,} {\bfseries 38} (1986) 672--696.

\end{thebibliography}
\end{document}